 \documentclass[final,3p,times]{elsarticle}

\usepackage{hyperref}
\usepackage{showkeys}
\usepackage{xspace}
\usepackage{subfig}
\usepackage{mathrsfs,stmaryrd,amsmath}
\usepackage{url}
\usepackage{xcolor}
\usepackage{enumitem}
\usepackage{tikz}
\usetikzlibrary{decorations.pathreplacing}
\usetikzlibrary{patterns}
\usepackage{amssymb,amsmath}
\usepackage{empheq}
\usepackage[titletoc]{appendix}
\setlist[enumerate]{label=\roman*)}

\makeatletter
\def\ps@pprintTitle{%
     \let\@oddhead\@empty
     \let\@evenhead\@empty
     \def\@oddfoot{\footnotesize\itshape
       Preprint \ifx\@journal\@empty
       \else\@journal\fi\hfill\today}%
     \let\@evenfoot\@oddfoot}
\makeatother

\newcommand{\spec}[1]{ \sigma(#1) }
\newcommand{\dom}[1]{ D(#1) }
\newcommand{\opA}{A}
\newcommand{\opB}{B}
\newcommand{\opAt}{\widetilde{A}}
\newcommand{\opBt}{\widetilde{B}}
\newcommand{\opSt}{\widetilde{S}}
\newcommand{\ut}{\widetilde{u}}
\newcommand{\Ut}{\widetilde{U}}
\newcommand{\uzt}{\widetilde{u}^0}

\newcommand{\range}{\mathrm{Im} \, }
\newcommand{\ps}[3]{ {\left\langle #1 , #2 \right\rangle}_{#3} }

\newcommand{\field}[1]{\ensuremath{\mathbb{#1}}}
\newcommand{\C}{\field{C}\xspace}
\newcommand{\R}{\field{R}\xspace}
\newcommand{\Z}{\field{Z}\xspace}
\newcommand{\N}{\field{N}\xspace}

\newcommand{\lin}[1]{\mathcal{L}(#1)}

\newcommand{\Span}{{\rm Span}}

\newcommand\Id{\mathrm{Id}}

\newcommand\ddt{\frac{d}{dt}}

\newcommand\ext[1]{\overline{#1}}

\newcommand\Tau{\mathcal{T}}

\newcommand\cinfc{{C}^{\infty}_{c}}

\newcommand\g{g}
\newcommand\G{G}
\newcommand\K{K}
\newcommand\h{h}

\renewcommand\k{k}
\renewcommand\H{H}

\def\ds{\displaystyle}
\def\norm#1{\left\|#1\right\|}

\newcommand{\abs}[1]{ \left|#1\right| }

\newcommand{\ens}[1]{ \left\{#1\right\} }

\newcommand\tih{\widetilde{k}}
\newcommand\tiU{\widetilde{U}}

\newcommand\bmu{(\opB^*)^{-1}}

\usepackage{amssymb}
 \usepackage{amsthm}

\newtheorem{theorem}{Theorem}[section]
 
 \newtheorem{lemma}{Lemma}[section]
 \newtheorem{proposition}{Proposition}[section]
 \newtheorem{definition}{Definition}[section]
 
\newtheorem{remark}{Remark}

 \numberwithin{equation}{section}

\begin{document}

\begin{frontmatter}


 
 \author[authorJMC]{Jean-Michel Coron\fnref{authJMC}}
 \ead{coron@ann.jussieu.fr}

 \author[authorLH,authorJMC]{Long Hu\fnref{LH}}
 \ead{hu@ann.jussieu.fr}

 \author[authorJMC]{Guillaume Olive\fnref{authGO}}
\ead{oliveg@ljll.math.upmc.fr}


 \address[authorJMC]{Sorbonne Universit\'{e}s, UPMC Univ Paris 06, UMR 7598, Laboratoire
Jacques-Louis Lions, 4 place Jussieu, 75252 Paris cedex 05, France.}
\address[authorLH]{School of Mathematics, Shandong University, Jinan, Shandong 250100, China. }

 \fntext[authJMC]{JMC was
supported by the ERC advanced grant 266907 (CPDENL) of the 7th Research Framework
Programme (FP7).}
 \fntext[LH]{LH was
supported by the ERC advanced grant 266907 (CPDENL) of the 7th Research Framework
Programme (FP7).}
\fntext[authGO]{GO was partially
supported by the ERC advanced grant 266907 (CPDENL) of the 7th Research Framework
Programme (FP7).}

\title{Stabilization and controllability of first-order integro-differential hyperbolic equations}

%

\begin{abstract}
In the present article we study the stabilization of first-order linear integro-differential hyperbolic equations.
For such equations we prove that the stabilization in finite time is equivalent to the exact controllability property.
The proof relies on a Fredholm transformation that maps the original system into a finite-time stable target system.
The controllability assumption is used to prove the invertibility of such a transformation.
Finally, using the method of moments, we show in a particular case that the controllability is reduced to the criterion of Fattorini.
\end{abstract}

\begin{keyword}
Integro-differential equation\sep Stabilization\sep Controllability\sep Backstepping\sep Fredholm integral.


\end{keyword}

\end{frontmatter}



\section{Introduction and main results}

The purpose of this article is the study of the stabilization and controllability properties of the equation
\begin{equation}\label{syst init}
\left\{
\begin{array}{rll}
u_t(t,x) -u_x(t,x) =&\ds \int_0^L \g(x,y)u(t,y) \, dy, & t \in (0,T), \, x \in (0,L), \\
u(t,L)=& U(t), & t \in (0,T), \\
u(0,x) =& u^0(x), & x \in (0,L).
\end{array}
\right.
\end{equation}
In \eqref{syst init}, $T>0$ is the time of control, $L>0$ the length of the domain.
$u^0$ is the initial data and $u(t,\cdot):[0,L] \longrightarrow \C$ is the state at time $t \in [0,T]$, $\g:(0,L)\times(0,L) \longrightarrow \C$ is a given function in $L^2((0,L)\times(0,L))$ and, finally, $U(t) \in \C$ is the boundary control at time $t \in (0,T)$.

The stabilization and controllability of \eqref{syst init} started in \cite{KS}.
The authors proved that the equation
$$
\left\{
\begin{array}{rll}
u_t(t,x) -u_x(t,x) =&\ds \int_0^x \g(x,y)u(t,y) \, dy+f(x)u(t,0), & t \in (0,T), \, x \in (0,L), \\
u(t,L)=& U(t), & t \in (0,T), \\
u(0,x) =& u^0(x), & x \in (0,L),
\end{array}
\right.
$$
with $\g$ and $f$ continuous, is always stabilizable in finite time (see also \cite{Nak} for the same equation with the nonlocal boundary condition $u(t,L)=\int_0^L u(t,y) \gamma(y) \,dy + U(t)$ with $\gamma$ continuous).
The proof uses the  backstepping approach introduced and developed by M.~Krstic and his co-workers (see, in particular,  the pioneer articles \cite{2003-Bokovic-Balogh-Krstic-MCSS, 2003-Liu-SICON, 2004-Smyshlyaev-Krstic-IEEE} and the reference book \cite{KSbook}).
This approach consists in  mapping \eqref{syst init} into the following finite-time stable target system
$$
\left\{
\begin{array}{rll}
w_t(t,x) -w_x(t,x) =& 0, & t \in (0,T), \, x \in (0,L), \\
w(t,L)=& 0, & t \in (0,T), \\
w(0,x) =& w^0(x), & x \in (0,L),
\end{array}
\right.
$$
by means of the Volterra transformation of the second kind
\begin{equation}\label{volt transfo}
u(t,x)=w(t,x)-\int_0^x \k(x,y)w(t,y)dy,
\end{equation}
where the kernel $\k$ has to satisfy some PDE in the triangle $0 \leq y \leq x \leq L$ with appropriate boundary conditions, the so-called kernel equation.
Let us emphasize that the strength of this method is that the Volterra transformation \eqref{volt transfo} is always invertible (see e.g. \cite[Chapter 2, THEOREM 6]{Hoc}).
Now, if the integral term is not anymore of Volterra type, that is if $\g$ in \eqref{syst init} does not satisfy
\begin{equation}\label{hyp KS}
\g(x,y)=0, \quad x \leq y,
\end{equation}
then, the Volterra transformation \eqref{volt transfo} can no longer be used (there is no solution to the kernel equation which is supported in the triangle $0 \leq y \leq x \leq L$ in this case, see the equation \eqref{equ hstar} below).
In \cite{BAK}, the authors suggested to replace the Volterra transformation \eqref{volt transfo} by the more general Fredholm transformation
\begin{equation}\label{fred transfo}
u(t,x)=w(t,x)-\int_0^L \k(x,y)w(t,y)dy,
\end{equation}
where $\k \in L^2((0,L)\times(0,L))$ is a new kernel.
However, the problem is now that, unlike the Volterra transformation \eqref{volt transfo}, the Fredholm transformation \eqref{fred transfo} is not always invertible.
In \cite{BAK}, the authors proved that, if $\g$ is small enough, then the transformation \eqref{fred transfo} is indeed invertible, see \cite[Theorem 9]{BAK}.
They also gave some sufficient conditions in the case $\g(x,y)=\g(y)$, see \cite[Theorem 1.11]{BAK}.
Our main result states that we can find a particular kernel $\k$ such that the corresponding Fredholm transformation \eqref{fred transfo} is invertible, if we assume that  \eqref{syst init} is exactly controllable at time $L$.
Finally, let us point out that Fredholm transformations have also been used to prove the exponential stabilization for a Korteweg-de Vries equation in \cite{2014-Coron-Lu-JMPA} and for a Kuramoto-Sivashinsky equation in \cite{2015-Coron-Lu-JDE}.
In these papers also, the existence of the kernel and the invertibility of the associated transformation were established under a controllability assumption.
However, our proof is of a completely different spirit than the one given in these articles.

\subsection{Well-posedness}

Multiplying formally \eqref{syst init} by the complex conjugate of a smooth function $\ext{\phi}$ and integrating by parts, we are lead to the following definition of solution:
\begin{definition}
Let $u^0 \in L^2(0,L)$ and $U \in L^2(0,T)$.
We say that a function $u$ is a (weak) solution to \eqref{syst init} if $u \in C^0([0,T];L^2(0,L))$ and 
\begin{multline}\label{weak sol}
\int_0^{\tau} \int_0^L u(t,x)\ext{\left(-\phi_t(t,x)+\phi_x(t,x)-\int_0^L \ext{\g(y,x)}\phi(t,y) \, dy\right)} \, dx dt
\\
+\int_0^L u(\tau,x) \ext{\phi(\tau,x)} \, dx
-\int_0^L u^0(x) \ext{\phi(0,x)} \, dx
-\int_0^{\tau} U(t)\ext{\phi(t,L)} \, dt=0,
\end{multline}
for every $\phi \in C^1([0,\tau]\times[0,L])$ such that $\phi(\cdot,0)=0$, and every $\tau \in [0,T]$.
\end{definition}

Let us recall that \eqref{syst init} can equivalently be rewritten in the abstract form
\begin{equation}\label{abst}
\left\{
\begin{array}{rll}
\ds \ddt u=& \opA u + \opB U, & t \in (0,T), \\
u(0) =& u^0, &
\end{array}
\right.
\end{equation}
where we can identify the operators $\opA$ and $\opB$ through their adjoints by taking formally the scalar product of \eqref{abst} with a smooth function $\phi$ and then comparing with \eqref{weak sol}.
The operator $\opA:\dom{\opA} \subset L^2(0,L) \longrightarrow L^2(0,L)$ is thus given by
\begin{equation}\label{def opA}
\opA u = \ds u_x+\int_0^L \g(\cdot,y)u(y) \, dy,
\end{equation}
with
$$\dom{\opA} = \ens{u \in H^1(0,L) \quad \middle| \quad u(L)=0}.$$
Clearly, $\opA$ is densely defined, and its adjoint $\opA^*: \dom{\opA^*} \subset L^2(0,L) \longrightarrow L^2(0,L)$ is
\begin{equation}\label{opAstar}\opA^* z= \ds -z_x+\int_0^L \ext{\g(y,\cdot)}z(y) \, dy,
\end{equation}
with
$$\dom{\opA^*}=\ens{z \in H^1(0,L) \quad \middle| \quad z(0)=0}.$$
Using the Lumer-Philips' theorem (see e.g. \cite[Chapter 1, Corollary 4.4]{Paz}), we can prove that $\opA$ generates a $C_0$-group ${(S(t))}_{t \in \R}$.

In particular, $\opA^*$ is closed and its domain $\dom{\opA^*}$ is then a Hilbert space, equipped with the scalar product associated with the graph norm $\norm{z}_{\dom{\opA^*}}=(\norm{z}_{L^2}^2+\norm{\opA^*z}_{L^2}^2)^{1/2}$, $z \in \dom{\opA^*}$.
Observe that
\begin{equation}\label{equiv norm}
\norm{\cdot}_{\dom{\opA^*}} \mbox{ and } \norm{\cdot}_{H^1(0,L)} \mbox{ are equivalent norms on } \dom{\opA^*}.
\end{equation}

On the other hand, the operator $\opB \in \lin{\C,\dom{\opA^*}'}$ is
\begin{equation}\label{def opB}
\ps{\opB U}{z}{\dom{\opA^*}',\dom{\opA^*}}= U \ext{z(L)}.
\end{equation}
Note that $\opB$ is well defined since $BU$ is continuous on $H^1(0,L)$ (by the trace theorem $H^1(0,L) \hookrightarrow C^0([0,L])$) and since we have \eqref{equiv norm}.
Its adjoint $\opB^* \in \lin{\dom{\opA^*},\C}$ is
\begin{equation}\label{opBstar}
\opB^*z=z(L).
\end{equation}
One can prove that $\opB$ satisfies the following so-called admissibility condition\footnote{The proof is analogous to the one of Lemma \ref{lemma 2 app} in \ref{app 2}.}:
\begin{equation}\label{admi}
\exists C>0, \quad \int_0^{T} \abs{\opB^*S(T-t)^*z}^2 \, dt \leq C \norm{z}_{L^2(0,L)}^2, \quad \forall z \in \dom{\opA^*}.
\end{equation}
Note that $\opB^*S(T-\cdot)^*z$ makes sense in \eqref{admi} since $S(T-\cdot)^*z \in \dom{\opA^*}$ for $z \in \dom{\opA^*}$, while it does not in general if $z$ is only in $L^2(0,L)$.
Thus, \eqref{admi} allows us to continuously extend in a unique way the map $z \longmapsto \opB^*S(T-\cdot)^*z$ to the whole space $L^2(0,L)$ and give in particular a sense to $\opB^*S(T-\cdot)^*z$ for $z \in L^2(0,L)$.
We shall keep the same notation to denote this extension.

Finally, we recall that, since $\opA$ generates a $C_0$-semigroup and $\opB$ is admissible, for every $u^0 \in L^2(0,L)$ and every $U \in L^2(0,T)$, there exists a unique solution $u \in C^0([0,T];L^2(0,L))$ to \eqref{syst init}.
Moreover, there exists $C>0$ (which does not depend on $u^0$ nor $U$) such that
$$\norm{u}_{C^0([0,T];L^2(0,L))} \leq C \left(\norm{u^0}_{L^2(0,L)}+\norm{U}_{L^2(0,T)}\right).$$
See e.g. \cite[Theorem 2.37]{Cor} and \cite[Section 2.3.3.1]{Cor}.

\subsection{Controllability and stabilization}

Let us now recall the definitions of the properties we are interested in.

\begin{definition}
We say that \eqref{syst init} is exactly controllable at time $T$ if, for every $u^0,u^1 \in L^2(0,L)$, there exists $U \in L^2(0,T)$ such that the corresponding solution $u$ to \eqref{syst init} satisfies
$$u(T)=u^1.$$
If the above property holds for $u^1=0$, we say that \eqref{syst init} is null-controllable at time $T$.
\end{definition}

\begin{remark}\label{rem nc}
Since $\opA$ generates a group, \eqref{syst init} is exactly controllable at time $T$ if, and only if, \eqref{syst init} is null-controllable at time $T$ (see e.g. \cite[Theorem 2.41]{Cor}).
\end{remark}

\begin{definition}\label{def stab}
We say that \eqref{syst init} is stabilizable in finite time $T$ if there exists a bounded linear map $\Gamma:L^2(0,L) \longrightarrow \C$ such that, for every $u^0 \in L^2(0,L)$, the solution $u \in C^0([0,+\infty);L^2(0,L))$ to
\begin{equation}\label{syst fb}
\left\{
\begin{array}{rll}
u_t(t,x) -u_x(t,x) =&\ds \int_0^L \g(x,y)u(t,y) \, dy, & t \in (0,+\infty), \, x \in (0,L), \\
u(t,L)=& \Gamma u(t), & t \in (0,+\infty), \\
u(0,x) =& u^0(x), & x \in (0,L),
\end{array}
\right.
\end{equation}
satisfies
\begin{equation}\label{uzat}
u(t)=0, \quad \forall t \geq T.
\end{equation}
\end{definition}

Note that \eqref{syst fb} is well-posed.
Indeed, by the Riesz representation theorem, there exists $\gamma \in L^2(0,L)$ such that
\begin{equation}\label{fb riesz}
\Gamma u=\int_0^L u(y)\ext{\gamma(y)} \, dy,
\end{equation}
and \eqref{syst fb} with \eqref{fb riesz} is well-posed (see e.g. \cite[Theorem 2.1]{Nak}).

\begin{remark}\label{rem stab nc}
Let us recall here some links between stabilization and controllability.
Clearly, stabilization in finite time $T$ implies null-controllability at time $T$.
It is also well-known that in finite dimension (that is when $\opA$ and $\opB$ are matrices) controllability is equivalent to exponential stabilization at any decay rate, see e.g. \cite[PART I, Theorem 2.9]{Zab}.
Finally, for bounded operators $\opB$ (which is not the case here though), null-controllability at some time implies exponential stabilization, see e.g. \cite[PART IV, Theorem 3.3]{Zab}.
We refer to \cite{JZ1999} and the references therein for recent results on the exponential stabilization of one-dimensional systems generated by $C_0$-groups (including then \eqref{syst init}) and to \cite{BT} for the exponential stabilization of systems generated by analytic $C_0$-semigroups.
\end{remark}

\subsection{Main results}

Let us introduce the triangles
$$\Tau_-=\ens{(x,y) \in (0,L) \times (0,L) \quad \middle| \quad x>y},
\qquad
\Tau_+=\ens{(x,y) \in (0,L) \times (0,L) \quad \middle| \quad x<y}.$$
For the stabilization, we will always assume that
\begin{equation}\label{reg G}
\g \in  H^1(\Tau_-) \cap H^1(\Tau_+).
\end{equation}
This means that we allow integral terms whose kernel has a discontinuity along the diagonal of the square $(0,L)\times(0,L)$:
$$\int_0^x \g_1(x,y)u(t,y) \, dy
+\int_x^L \g_2(x,y)u(t,y) \, dy,$$
with $\g_1,\g_2 \in H^1((0,L)\times(0,L))$.
We gathered in \ref{app 1} some properties of the functions of $H^1(\Tau_-) \cap H^1(\Tau_+)$.

Our main result is then the following:
\begin{theorem}\label{thm stab}
Assume that \eqref{reg G} holds.
Then, \eqref{syst init} is stabilizable in finite time $L$ if, and only if, \eqref{syst init} is exactly controllable at time $L$.
\end{theorem}
(note that the necessary part is clear from Remark \ref{rem stab nc} and Remark \ref{rem nc}).

Thus, we see that we have to study the controllability of \eqref{syst init} at the optimal time of control $T=L$ (we recall that, in the case $\g=0$, \eqref{syst init} is exactly controllable at time $T$ if, and only if, $T \geq L$).
We will show that this property is characterized by the criterion of Fattorini in the particular case
\begin{equation}\label{hyp gdx}
\g(x,y)=\g(x), \quad \g \in L^2(0,L).
\end{equation}

Indeed, the second result of this paper is
\begin{theorem}\label{thm cont}
Assume that \eqref{hyp gdx} holds.
Then, \eqref{syst init} is exactly controllable at time $L$ if, and only if,
\begin{equation}\label{ucp}
\ker(\lambda-\opA^*) \cap \ker \opB^*=\ens{0}, \quad \forall \lambda \in \C.
\end{equation}
\end{theorem}

Actually, we conjecture that Theorem \ref{thm cont} remains true without assuming \eqref{hyp gdx}.

\begin{remark}\label{rem ac}
In fact, \eqref{ucp} is a general necessary condition for the approximate controllability.
Let us recall that we say that \eqref{syst init} is approximately controllable at time $T$ if, for every $\epsilon>0$, for every $u^0,u^1 \in L^2(0,L)$, there exists $U \in L^2(0,T)$ such that the corresponding solution $u$ to \eqref{syst init} satisfies
$$\norm{u(T)-u^1}_{L^2(0,L)} \leq \epsilon.$$
Clearly, it is a weaker property than exact controllability.
Let us also recall that this property is equivalent to the following dual one (see e.g. \cite[Theorem 2.43]{Cor}):
\begin{equation}\label{ucp full}
\forall z \in L^2(0,L), \quad \Big(\opB^*S(t)^*z=0 \mbox{ for a.e. } t \in (0,T) \Big) \Longrightarrow z=0.
\end{equation}
Thus, we see that \eqref{ucp} is nothing but the property \eqref{ucp full} only for $z \in \ker(\lambda-\opA^*)$ since $S(t)^*z=e^{\lambda t}z$ for $z \in \ker(\lambda-\opA^*)$.
This condition \eqref{ucp} is misleadingly known as the Hautus test \cite{Hau} in finite dimension, despite it has been introduced earlier by H.O.~Fattorini in \cite{Fat} and in a much larger setting.
Finally, let us mention that it has also been proved in \cite{BT} that \eqref{ucp} characterizes the exponential stabilization of parabolic systems.
\end{remark}

\begin{remark}
We will exhibit functions $\g$ such that \eqref{ucp} does not hold for an arbitrary large number of $\lambda$, see Remark \ref{rem exist G} below.
On the other hand, we can check that \eqref{ucp} is satisfied for any $g \in L^2((0,L)\times(0,L))$ satisfying one of the following conditions:
\begin{enumerate}
\item
$\opA^*$ has no eigenvalue (as it is the case when $\g=0$).
\item\label{g small}
$\g$ is small enough: $\norm{g}_{L^2}<\frac{\sqrt{2}}{L}$.
\item\label{g volt}
$\g$ is of Volterra type (that is it satisfies \eqref{hyp KS}).
\end{enumerate}
The point \ref{g small} follows from the invertibility of transformations $\Id-\G$ for $\norm{\G}_{\lin{L^2}}<1$.
The point \ref{g volt} follows from the invertibility of Volterra operators.
\end{remark}

Let us notice that we can also consider equations of the more general form
$$
\left\{
\begin{array}{rll}
\ut_t(t,x) -\ut_x(t,x) =&\ds  \int_0^L \tilde{\g}(x,y)\ut(t,y) \, dy + f(x)\ut(t,0)+d(x)\ut(t,x), & t \in (0,T), \, x \in (0,L), \\
\ut(t,L)=& \ds \int_0^L  \ut(t,y) \gamma(y)\, dy
+ \Ut(t), & t \in (0,T), \\
\ut(0,x) =& \uzt(x), & x \in (0,L),
\end{array}
\right.
$$
where $f,d,\gamma:(0,L) \longrightarrow \C$ and $\tilde{\g}:(0,L)\times(0,L) \longrightarrow \C$ are regular enough.
Performing a transformation of Volterra type, it can actually be reduced to an equation like \eqref{syst init}.
See \cite[Theorem 3.2]{Nak} for more details.

Let us conclude the introduction by pointing out that Theorem \ref{thm stab} still holds if we consider
states and controls taking their values into $\R$ instead of $\C$ provided that
\begin{equation}\label{g(x,y)real}
g(x,y)\in \R \, \mbox{ for a.e. } (x,y)\in (0,L)\times (0,L).
\end{equation}
This follows from the fact that, if \eqref{g(x,y)real} holds and if the control system \eqref{syst init}, with real valued states and controls, is exactly controllable at time $L$, then the functions $\k$ and $U$ constructed in the proof of Proposition \ref{prop exist k} below are real valued functions.
Concerning Theorem \ref{thm cont}, it also still holds for real valued states and controls if $g$ is real valued (but, of course, we still have to consider  in \eqref{ucp} complex valued functions and complex $\lambda$).

\section{Finite-time stabilization}\label{sect stab}

\subsection{Presentation of the method}

Let us write $\opA=A_0+G$ where the unbounded linear operator $A_0:\dom{A_0} \subset L^2(0,L) \longrightarrow L^2(0,L)$ is defined by
$$A_0 u = \ds u_x, \quad \dom{A_0} =\dom{\opA},$$
and the bounded linear operator $\G: L^2(0,L) \longrightarrow L^2(0,L)$ is defined by
$$\G u = \ds \int_0^L \g(\cdot,y)u(y) \, dy.$$
Note that the adjoint $\opA_0^*:\dom{\opA_0^*}\subset L^2(0,L) \longrightarrow L^2(0,L)$ of $\opA_0$ is the operator
$$\opA_0^*z=-z_x, \quad \dom{\opA_0^*} =\ens{z \in H^1(0,L) \quad \middle| \quad z(0)=0}.$$

We  first perform some formal computations to explain the ideas of our method.
We recall that the strategy is to map the initial equation
\begin{equation}\label{formal fb}
\left\{
\begin{array}{rll}
\ds \ddt u=& (\opA+\opB\Gamma)u, & t \in (0,+\infty), \\
u(0) =& u^0, &
\end{array}
\right.
\end{equation}
into the finite-time stable target equation
\begin{equation}\label{syst targ}
\left\{
\begin{array}{rll}
\ds \ddt w=& A_0 w, & t \in (0,+\infty), \\
w(0) =& w^0, &
\end{array}
\right.
\end{equation}
for some operator $\Gamma$ and by means of a transformation $P$ (independent of the time $t$):
$$u=Pw.$$
If $u=Pw$ where $w$ solves \eqref{syst targ}, then
\begin{equation}\label{formal 1}
\ddt u=\ddt \left(Pw\right)=P\left(\ddt w\right)=PA_0 w,
\end{equation}
and
\begin{equation}\label{formal 2}
(\opA+\opB\Gamma)u
=(\opA P+\opB\Gamma P)w.
\end{equation}
As a result, $u$ solves \eqref{formal fb} if the right-hand sides of \eqref{formal 1} and \eqref{formal 2} are equals, that is, if $P$ and $\Gamma$ satisfy
$$PA_0=\opA P+\opB \Gamma P.$$
Taking the adjoints, this is equivalent to
\begin{equation}\label{rela fonda}
A_0^*P^*=P^*\opA^*+P^*\Gamma^*\opB^*.
\end{equation}
By \eqref{rela fonda}, we mean that
\begin{empheq}[left=\empheqlbrace]{align}
&P^*\left(\dom{\opA^*}\right) \subset \dom{\opA^*_0}=\dom{\opA^*}, \label{P regul}
\\
& \opA_0^*P^*z=P^*\opA^*z+P^*\Gamma^*\opB^*z,\quad \forall z \in \dom{\opA^*}. \label{rela fonda-explicit}
\end{empheq}

The following proposition gives the rigorous statement of what we have just discussed (the proof is given in  \ref{app proof prop}).
\begin{proposition}\label{prop equiv sol}
Assume that there exist a bounded linear operator $P:L^2(0,L) \longrightarrow L^2(0,L)$ and a bounded linear form $\Gamma: L^2(0,L) \longrightarrow \C$ such that:
\begin{enumerate}
\item
\eqref{P regul}-\eqref{rela fonda-explicit} hold.

\item
$P$ is invertible.
\end{enumerate}
Then, for every $u^0 \in L^2(0,L)$, if $w \in C^0([0,T];L^2(0,L))$ denotes the solution to \eqref{syst targ} with $w^0=P^{-1}u^0$, then $u=Pw$ is the solution to \eqref{syst fb} and it satisfies \eqref{uzat}.
\end{proposition}

Let us now "split" the equation \eqref{rela fonda-explicit}.
We recall that $\dom{\opA^*}$ is a Hilbert space and $\opB^*$ is continuous for the norm of $\dom{\opA^*}$ (see the introduction).
Thus, its kernel $\ker \opB^*$ is closed for this norm and we can write the orthogonal decomposition
$$\dom{\opA^*}=\ker \opB^* \oplus \left(\ker \opB^*\right)^{\perp},$$
where $V^{\perp}$ denotes the orthogonal of a subspace $V$ in $\dom{\opA^*}$.
Noting that $\opB^*$ is a bijection from $\left(\ker \opB^*\right)^{\perp}$ to $\C$ (with inverse denoted by $\bmu$), we see that \eqref{rela fonda-explicit} holds if, and only if,
\begin{equation}\label{equ kerb}
A_0^*P^*z-P^*\opA^*z=0, \quad \forall z \in \ker \opB^*,
\end{equation}
and
\begin{equation}\label{equ kerb ort}
P^*\Gamma^*=\left(A_0^*P^*-P^*\opA^*\right)\bmu.
\end{equation}

It follows from this observation that it is enough to establish the existence of $P$ such that \eqref{equ kerb} hold and $P$ is invertible.
The map $\Gamma$ will then be defined as the adjoint of the linear map $\Psi: \C \longrightarrow L^2(0,L)$ defined by
\begin{equation}\label{def gamma abstr}
\Psi =\left((P^*)^{-1}A_0^*P^*-\opA^*\right)\bmu.
\end{equation}
Note that $P^*:\dom{\opA^*}\longrightarrow \dom{\opA^*}$ is continuous by the closed graph theorem, so that $\Psi$ defined by \eqref{def gamma abstr} is bounded.

Let us summarize the discussion:
\begin{proposition}
Let $P:L^2(0,L) \longrightarrow L^2(0,L)$ be a bounded linear operator such that \eqref{P regul} holds and $P$ is invertible.
Then, there exists a bounded linear form $\Gamma: L^2(0,L) \longrightarrow \C$ such that \eqref{rela fonda-explicit} holds if, and only if, $P^*$ satisfies \eqref{equ kerb}.
\end{proposition}

A discussion on other expressions of $\Gamma$ than \eqref{def gamma abstr} is given in Section \ref{sect gamma} below.

\subsection{Construction of the transformation}

In this section, we are going to construct a map $P$ such that \eqref{P regul} and \eqref{equ kerb} hold.
We look for $P$ in the form
\begin{equation}\label{P}
P=\Id-\K,
\end{equation}
where $\K:L^2(0,L) \longrightarrow L^2(0,L)$ is an integral operator defined by
$$\K z(x) = \int_0^L \k(x,y)z(y) \, dy,$$
with $\k \in L^2((0,L)\times(0,L))$.
Clearly, its adjoint is
$$\K^*z(x)=\int_0^L \k^*(x,y) z(y) \, dy,$$
where we set
$$\k^*(x,y) = \ext{\k(y,x)}.$$
Let us recall that $\K$, as well as $\K^*$, is compact on $L^2(0,L)$.

For the expression \eqref{P}, \eqref{P regul} now read as
\begin{equation}\label{H regul}
\K^*\left(\dom{\opA^*}\right) \subset \dom{\opA^*},
\end{equation}
and \eqref{equ kerb} becomes
\begin{equation}\label{rela fonda 2}
-A_0^*\K^*z+\K^*A_0^*z+\K^*\G^*z-\G^*z=0, \quad \forall z \in \ker \opB^*.
\end{equation}
Let us now translate these properties in terms of the kernel $\k^*$.

\begin{proposition}\label{prop equiv H h}
Assume that
\begin{equation}\label{reg hstar}
\begin{array}{c}
\ds \k^* \in H^1(\Tau_-) \cap H^1(\Tau_+),
\end{array}
\end{equation}
and let $\k^*_+ \in L^2(\partial\Tau_+)$ be the trace on $\Tau_+$ of the restriction of $\k^*$ to $\Tau_+$.
Then,
\begin{enumerate}
\item
\eqref{H regul} holds if, and only if,
\begin{equation}\label{H regul traduit}
\k^*_+(0,y)=0, \quad y \in (0,L).
\end{equation}
\item
\eqref{rela fonda 2} holds if, and only if,
\begin{equation}\label{equ hstar}
\k^*_x(x,y)+\ds \k^*_y(x,y) +\int_0^L \ext{\g(y,\sigma)} \k^*(x,\sigma)d\sigma-\ext{\g(y,x)}=0, \quad x,y \in (0,L).
\end{equation}
\end{enumerate}
\end{proposition}

Observe that if $\k^* \in H^1(\Tau_-) \cap H^1(\Tau_+)$, then $\k^*_x, \k^*_y\in L^2((0,L)\times(0,L))$ and \eqref{equ hstar} is understood as an equality for almost every $(x,y) \in (0,L)\times(0,L)$.

\begin{proof}
Let us first prove the equivalence between \eqref{H regul} and \eqref{H regul traduit}.
Since $\k^* \in H^1(\Tau_-) \cap H^1(\Tau_+)$, we have $\K^*z \in H^1(0,L)$ for every $z \in L^2(0,L)$ with (see Proposition \ref{prop ipp} \ref{hyp A3})
$$\K^*z(0)=\int_0^L \k^*_+(0,y) z(y) \, dy.$$
Thus,  \eqref{H regul} holds if, and only if, $\K^*z(0)=0$ for every $z \in \dom{\opA^*}$, which gives \eqref{H regul traduit} by density of $\dom{\opA^*}$ in $L^2(0,L)$.

Let us now establish the equivalence between \eqref{rela fonda 2} and \eqref{equ hstar}.
Let us compute each terms in the left-hand side of \eqref{rela fonda 2} for any $z \in \dom{\opA^*}$.
For the first term we have (see Proposition \ref{prop ipp} \ref{hyp A3})
$$
\begin{array}{rl}
-A_0^*\K^*z(x) &\ds =
\partial_x\left(\int_0^L \k^*(\cdot,y)z(y) \, dy\right)(x) \\
&\ds =\Big(\k_-^*(x,x)-\k_+^*(x,x)\Big)z(x)
+\int_0^L \k^*_x(x,y)z(y) \, dy,
\end{array}
$$
where $\k^*_- \in L^2(\partial\Tau_-)$ (\textit{resp.} $\k^*_+ \in L^2(\partial\Tau_+)$) denotes the trace on $\Tau_-$ (\textit{resp.} $\Tau_+$) of the restriction of $\k^*$ to $\Tau_-$ (\textit{resp.} $\Tau_+$).
On the other hand, integrating by parts the second term and using $z(0)=0$ (since $z \in \dom{\opA^*}$), we have (see Proposition \ref{prop ipp}  \ref{hyp A2})
$$
\begin{array}{rl}
\K^*A_0^*z(x) &\ds =
-\int_0^L \k^*(x,y)z'(y) \, dy \\
&\ds =\int_0^L \k^*_y(x,y)z(y) \, dy
-\Big(\k_-^*(x,x)-\k_+^*(x,x)\Big)z(x)
-\k^*_+(x,L)z(L).
\end{array}
$$
Finally, the remaining term gives
$$
\begin{array}{rl}
\K^*\G^*z(x)-\G^*z(x) &\ds =
\int_0^L \k^*(x,y)\left(\int_0^L \ext{\g(\sigma,y)} z(\sigma) \, d\sigma\right) \, dy
-\int_0^L \ext{\g(y,x)} z(y) \, dy
\\
&\ds =\int_0^L \left(\int_0^L \k^*(x,\sigma)\ext{\g(y,\sigma)} \, d\sigma
-\ext{\g(y,x)}\right)z(y) \, dy.
\end{array}
$$
As a result, summing all the previous equalities, we have
\begin{multline}\label{equ K k}
-A_0^*\K^*z(x)+\K^*\opA_0^*z(x)+\K^*\G^*z(x)-\G^*z(x)
= \\
\int_0^L \left(\k^*_x(x,y)+\k^*_y(x,y)+\int_0^L \k^*(x,\sigma)\ext{\g(y,\sigma)} \, d\sigma
-\ext{\g(y,x)}\right)z(y) \, dy 
-\k^*_+(x,L)z(L),
\end{multline}
for every $z \in \dom{\opA^*}$.
In particular, we obtain that \eqref{rela fonda 2} is equivalent to
$$
\int_0^L \left(\k^*_x(x,y)+\k^*_y(x,y)+\int_0^L \k^*(x,\sigma)\ext{\g(y,\sigma)} \, d\sigma
-\ext{\g(y,x)}\right)z(y) \, dy=0,
$$
for every $z \in \ker \opB^*=H^1_0(0,L)$.
Since $H^1_0(0,L)$ is dense in $L^2(0,L)$, this is equivalent to the equation \eqref{equ hstar}.
\end{proof}

\begin{remark}\label{rem H regul}
In the first step of the proof we have in fact establish that \eqref{H regul traduit} is equivalent to
\begin{equation}\label{H regul 2}
\K^*\left(L^2(0,L)\right) \subset \dom{\opA^*}.
\end{equation}
We see that the operator $\K^*$ has a regularizing effect (under assumption \eqref{reg hstar}).
\end{remark}

\subsubsection{Existence of the kernel}

Viewing $x$ as the time parameter in \eqref{H regul traduit}-\eqref{equ hstar}, it is clear that these equations have at least one solution $\k^* \in C^0([0,L];L^2(0,L))$, if we add any artificial $L^2$ boundary condition at $(x,0)$.
In this section, we fix a particular boundary condition such that $\k^*$ satisfies, in addition, the final condition
\begin{equation}\label{hstar final}
\k^*(L,y)=0, \quad y \in (0,L).
\end{equation}
This property will be used to establish the invertibility of the Fredholm transformation associated with this $\k^*$, see Section \ref{sect inv fred} below.

\begin{proposition}\label{prop exist k}
Assume that \eqref{syst init} is exactly controllable at time $L$.
Then, there exists $U \in L^2(0,L)$ such that the solution $\k^* \in C^0([0,L];L^2(0,L))$ to
\begin{align}\label{equ hstar bis}
\begin{cases}
\k^*_x(x,y)+\ds \k^*_y(x,y)+\int_0^L \ext{\g(y,\sigma)} \k^*(x,\sigma)d\sigma-\ext{\g(y,x)}=0, \quad x,y \in (0,L),\\
\ds \k^*(x,L)=U(x), \quad x \in (0,L),\\
\ds \k^*(L,y)=0, \quad y \in (0,L),
\end{cases}
\end{align}
satisfies
\begin{equation}\label{hstar init}
\k^*(0,y)=0, \quad y \in (0,L).
\end{equation}
\end{proposition}

\begin{proof}
Since $x$ plays the role of the time, let us introduce
$$\tih(t,y)=\ext{\k^*(L-t,y)}.$$
Thus, we want to prove that there exists $\tiU \in L^2(0,L)$ such that the corresponding solution $\tih \in C^0([0,L];L^2(0,L))$ to
\begin{align}\label{equ tih}
\begin{cases}
\ds \tih_t(t,y)-\tih_y(t,y)=\int_0^L \g(y,\sigma) \tih(t,\sigma)d\sigma-\g(y,L-t), \quad t,y \in (0,L),\\
\ds \tih(t,L)=\tiU(t), \quad t \in (0,L), \\
\ds \tih(0,y)=0, \quad y \in (0,L),
\end{cases}
\end{align}
satisfies
\begin{equation}\label{tihfinal}
\tih(L,y)=0, \quad y \in (0,L).
\end{equation}
This is a control problem, which has a solution by assumption.
Indeed, let $p \in C^0([0,L];L^2(0,L))$ be the free solution to the nonhomogeneous equation
\begin{align*}
\begin{cases}
\ds p_t(t,y)-p_y(t,y)=\int_0^L \g(y,\sigma) p(t,\sigma)d\sigma-\g(y,L-t), \quad t,y \in (0,L),\\
\ds p(t,L)=0, \quad t \in (0,L), \\
\ds p(0,y)=0, \quad y \in (0,L),
\end{cases}
\end{align*}
and let $q \in C^0([0,L];L^2(0,L))$ be the controlled solution going from $0$ to $-p(L,\cdot)$:
\begin{align*}
\begin{cases}
\ds q_t(t,y)-q_y(t,y)=\int_0^L \g(y,\sigma) q(t,\sigma)d\sigma, \quad t,y \in (0,L),\\
\ds q(t,L)=\tiU(t), \quad t \in (0,L), \\
\ds q(0,y)=0, \quad q(L,y)=-p(L,y), \quad y \in (0,L).
\end{cases}
\end{align*}
Then, the function $\tih \in C^0([0,L];L^2(0,L))$ defined by
$$\tih=p+q,$$
satisfies \eqref{equ tih}-\eqref{tihfinal}.
\end{proof}

\subsubsection{Regularity of the kernel}

The next step is to establish the regularity \eqref{reg hstar} for $\k^*$ provided by Proposition \ref{prop exist k}.

\begin{proposition}\label{prop reg h}
Let $U \in L^2(0,L)$ and let $\k^* \in C^0([0,L];L^2(0,L))$ be the corresponding solution to \eqref{equ hstar bis}.
If $\k^*$ satisfies \eqref{hstar init} and \eqref{reg G} holds, then
$$U \in H^1(0,L), \quad \k^* \in H^1(\Tau_-) \cap H^1(\Tau_+).$$
\end{proposition}

The proof of Proposition \ref{prop reg h} relies on the following lemma:
\begin{lemma}\label{lem formul}
Let $f \in L^2((0,L)\times(0,L))$, $V \in L^2(0,L)$ and $v^0 \in L^2(0,L)$.
\begin{enumerate}
\item
The unique solution $v \in C^0([0,L];L^2(0,L))$ to
\begin{equation}\label{equ v}
\begin{cases}
\ds v_x(x,y)+v_y(x,y)=f(x,y), \quad x,y \in (0,L),\\
\ds v(x,L)=V(x), \quad x \in (0,L), \\
\ds v(L,y)=v^0(y), \quad y \in (0,L),
\end{cases}
\end{equation}
is given by
$$
v(x,y)=
\left\{\begin{array}{ll}
\ds V(L+x-y)-\int_x^{L+x-y} f(s,s+y-x) \, ds, & \mbox{ if } (x,y) \in \Tau_+, \\
\ds v^0(L+y-x)-\int_x^{L} f(s,s+y-x) \, ds, & \mbox{ if } (x,y) \in \Tau_-.
\end{array}\right.
$$

\item\label{hyp 02}
If $V \in H^1(0,L)$ (\textit{resp.} $v^0 \in H^1(0,L)$) and $f_y \in L^2(\Tau_+)$ (\textit{resp.} $f_y \in L^2(\Tau_-)$), then $v \in H^1(\Tau_+)$ (\textit{resp.} $v \in H^1(\Tau_-)$).

\item\label{hyp 03}
If $f_y \in L^2(\Tau_+)$ and $v(0,\cdot) \in H^1(0,L)$, then $V \in H^1(0,L)$.
\end{enumerate}
\end{lemma}

\begin{proof}
Let us apply Lemma \ref{lem formul} with $V=U \in L^2(0,L)$, $v^0=0$ and
$$
\begin{array}{c}
\ds f(x,y)=f_1(x,y)+f_2(x,y), \\
\ds f_1(x,y)=-\int_0^L \ext{\g(y,\sigma)} \k^*(x,\sigma)d\sigma, \quad
f_2(x,y)=\ext{\g(y,x)}.
\end{array}
$$
Since $\k^*,g \in L^2((0,L)\times(0,L))$, we have $f_1,f_2 \in L^2((0,L)\times(0,L))$.
By uniqueness, the corresponding solution $v$ to \eqref{equ v} is equal to $\k^*$.
Since $\g \in H^1(\Tau_-) \cap H^1(\Tau_+)$ by assumption \eqref{reg G}, we have $(f_2)_y \in L^2(\Tau_+)$ and $(f_2)_y \in L^2(\Tau_-)$ by definition.
On the other hand, for a.e. $x \in (0,L)$, the map $f_1(x):y\longmapsto f_1(x,y)$ belongs to $H^1(0,L)$ with derivative (see Proposition \ref{prop ipp}  \ref{hyp A3})
$$
f_1(x)'(y)=-\int_0^L \ext{\g_x(y,\sigma)}\k^*(x,\sigma) \, d\sigma
-\Big(\g_-(y,y)-\g_+(y,y)\Big)\k^*(x,y).
$$
This shows that $(f_1)_y \in L^2(\Tau_-)$ and $(f_1)_y \in L^2(\Tau_+)$ (see Proposition \ref{prop caract hut}).
Finally, since $\k^*$ satisfies $\k^*(0,y)=0$ for a.e. $y \in (0,L)$, by Lemma \ref{lem formul} \ref{hyp 03} we have $U \in H^1(0,L)$.
Then, it follows from Lemma \ref{lem formul} \ref{hyp 02} that $\k^* \in H^1(\Tau_-) \cap H^1(\Tau_+)$.
\end{proof}

\subsection{Invertibility of the transformation}\label{sect inv fred}

To conclude the whole proof of Theorem \ref{thm stab}, it only remains to establish the invertibility of the transformation $\Id-\K^*$ with $\k^*$ provided by Proposition \ref{prop exist k}.
Let us start with a general lemma on the injectivity of maps $P^*$ for $P$ satisfying \eqref{P regul}-\eqref{equ kerb}.

\begin{lemma}\label{prop inv}
Let $P:L^2(0,L) \longrightarrow L^2(0,L)$ be a bounded linear operator such that \eqref{P regul}-\eqref{equ kerb} hold.
Then, we have
$$\ker P^*=\ens{0},$$
if, and only if, the following four conditions hold:
\begin{enumerate}
\item\label{hyp 1}
$\ker P^* \subset  \dom{\opA^*}$.

\item\label{hyp 1b}
$\ker P^* \subset \ker \opB^*$.

\item\label{hyp 2}
$\dim \ker P^* <+\infty$.

\item\label{hyp 3}
$\ker(\lambda-\opA^*) \cap \ker \opB^*=\ens{0}$ for every $\lambda \in \C$.
\end{enumerate}
\end{lemma}

\begin{proof}
Let us denote
$$N=\ker P^*.$$
Assume first that \ref{hyp 1}, \ref{hyp 1b}, \ref{hyp 2} and \ref{hyp 3} hold.
We want to prove that $N=\ens{0}$.
We argue by contradiction: assume that $N \neq \ens{0}$.
Let us prove that $N$ is stable by $\opA^*$.
By \ref{hyp 1} we have $N \subset \dom{\opA^*}$.
Let then $z \in N$ and let us show that $\opA^*z \in N$.
Since $N \subset \ker \opB^*$ by \ref{hyp 1b}, we can apply \eqref{equ kerb} to $z$ and obtain
$$P^*\opA^*z=A_0^*P^*z.$$
Since $z \in \ker P^*$ by definition, this gives
$$P^*\opA^*z=0,$$
and shows that $\opA^*z \in \ker P^*=N$.
Consequently, the restriction ${\opA^*}_{|N}$ of $\opA^*$ to $N$ is a linear operator from $N$ to $N$.
Since $N$ is finite dimensional by \ref{hyp 2} and $N \neq \ens{0}$, ${\opA^*}_{|N}$ has at least one eigenvalue $\lambda \in \C$.
Let $\xi \in N$ be a corresponding eigenfunction.
Thus,
$$\xi \in \ker(\lambda-\opA^*) \cap \ker \opB^*,$$
but
$$\xi \neq 0,$$
which is a contradiction with \ref{hyp 3}.
As a result, we must have $N=\ens{0}$.

Conversely, assume now that $\ker P^*=\ens{0}$.
It is clear that \ref{hyp 1}, \ref{hyp 1b} and \ref{hyp 2} hold.
Let $\lambda \in \C$ and $z \in \ker(\lambda-\opA^*) \cap \ker \opB^*$.
We want to prove that $z=0$.
By \eqref{equ kerb}, we have
$$A_0^*P^*z=\lambda P^*z,$$
that is
$$(\lambda- A_0^*)P^*z=0.$$
Since $\lambda-A_0^*$ (with domain $\dom{\opA_0^*}$) is injective and so is $P^*$ by assumption, this gives $z=0$.
\end{proof}

\begin{proposition}
Assume that \eqref{syst init} is exactly controllable at time $L$ and that \eqref{reg G} holds.
Then, the map $\Id-\K^*$, with $\k^*$ provided by Proposition \ref{prop exist k}, is invertible.
\end{proposition}

\begin{proof}
Since $\K^*$ is a compact operator, by Fredholm alternative it is equivalent to prove that $\Id-\K^*$ is injective.
In addition, the Fredholm alternative also gives
$$\dim \ker(\Id-\K^*) <+\infty.$$
Since $\Id-\K^*$ satisfies \eqref{H regul}-\eqref{rela fonda 2}, by Lemma \ref{prop inv} it is then equivalent to establish that
$$\ker(\Id-\K^*) \subset  \dom{\opA^*}, \quad \ker(\Id-\K^*) \subset \ker \opB^*.$$
The first inclusion follows from Remark \ref{rem H regul} and the second inclusion follows from the fact that
$$\opB^*\K^*z=0, \quad \forall z \in L^2(0,L),$$
which is equivalent to the condition \eqref{hstar final}.
\end{proof}

\subsection{Feedback control law}\label{sect gamma}

The proof of Theorem \ref{thm stab} is by now complete but we want to give a more explicit formula for $\Gamma$.
We recall that its adjoint $\Gamma^*$ is given by (see \eqref{def gamma abstr})
$$\Gamma^*=(P^*)^{-1}\left(A_0^*P^*-P^*\opA^*\right)\bmu.$$
Actually, we already computed $A_0^*P^*z-P^*\opA^*z$ for any $z \in \dom{\opA^*}$ in \eqref{equ K k} and we obtained that
$$A_0^*P^*z-P^*\opA^*z=-\k^*_+(\cdot,L)z(L).$$
Thus,
$$P^*\Gamma^*a=-\k^*_+(\cdot,L)a, \quad a \in \C.$$
Computing the adjoints, we obtain
$$\Gamma u=-\int_0^L \k_-(L,x) P^{-1} u(x) \, dx, \quad u \in L^2(0,L).$$
It is interesting to see that the open loop control $U$ provided by Proposition \ref{prop exist k} defines the closed loop control $\Gamma$ (since $\k_-(L,x)=\ext{U(x)}$ for a.e. $x \in (0,L)$).

Let us now recall that $P$ is of the form $P=\Id-\K$ and that the inverse of such an operator is also of the form $\Id-\H$ (with $\H=-(\Id-\K)^{-1}\K$).
Moreover, since $\K$ is an integral operator so is $\H$, with kernel $\h(\cdot,y)=-(\Id-\K)^{-1}\k(\cdot,y)$.
We can check that $\h$ inherits the regularity of $\k$ and satisfies a similar equation:
\begin{align*}
\begin{cases}
\h_x(x,y)+\ds \h_y(x,y)-\int_0^L \g(\sigma,y) \h(x,\sigma)d\sigma+\g(x,y)=0, \quad x,y \in (0,L),\\
\ds \h(x,0)=0, \quad \h(x,L)=0, \quad x \in (0,L).
\end{cases}
\end{align*}
Finally, a simple computation shows that $\Gamma$ is given by
$$\Gamma u =\int_0^L \h_-(L,y)u(y) \, dy,$$
where $\h_- \in L^2(\partial\Tau_-)$ denotes the trace on $\Tau-$ of the restriction of $\h$ to $\Tau_-$.

\section{Controllability}\label{sect cont}

The aim of this section is to study the controllability properties of \eqref{syst init} at the optimal time $T=L$ to provide easily checkable conditions to apply Theorem \ref{thm stab}.
Let us first mention that the controllability of one-dimensional systems generated by $C_0$-groups has already been investigated in a series of papers \cite{JZ2001} and \cite{GX}.
However, all these papers do not really focus on the optimal time of controllability, which is crucial to apply our stabilization theorem.
Let us also point out that the method developped in \cite{LR} seems ineffective because of the integral term $\int_x^L \g(x,y)u(t,y) \, dy$ in \eqref{syst init}.
Finally, let us mention the result \cite[Theorem 2.6]{Meh} for the distributed controllability of compactly perturbated systems (the case of the optimal time can not be treated though).

In order to have a good spectral theory, we consider system \eqref{syst init} with periodic boundary conditions:

\begin{equation}\label{syst init p}
\left\{
\begin{array}{rll}
\ut_t(t,x) -\ut_x(t,x) =&\ds \int_0^L \g(x,y)\ut(t,y) \, dy, & t \in (0,T), \, x \in (0,L), \\
\ut(t,L)-\ut(t,0)=& \Ut(t), & t \in (0,T), \\
\ut(0,x) =& \uzt(x), & x \in (0,L),
\end{array}
\right.
\end{equation}
where $\uzt \in L^2(0,L)$ and $\Ut \in L^2(0,T)$.
In the abstract form, \eqref{syst init p} reads
$$
\left\{
\begin{array}{rll}
\ds \ddt \ut=& \opAt \ut + \opBt \Ut, & t \in (0,T), \\
\ut(0) =& \uzt, &
\end{array}
\right.
$$
where $\opAt$ is the operator $\opA$ (see \eqref{def opA}) but now with domain
$$\dom{\opAt} = \ens{\ut \in H^1(0,L) \quad \middle| \quad \ut(L)=\ut(0)},$$
and $\opBt$ is the operator $\opB$ (see \eqref{def opB}) but now considered as an operator of $\lin{\C,\dom{\opAt^*}'}$.
The adjoints of these operators also remain unchanged (see \eqref{opAstar} and \eqref{opBstar}), except for their domain:
$$\dom{\opAt^*}=\dom{\opAt}, \quad \opBt^* \in \lin{\dom{\opAt^*},\C}.$$
Once again, we can check that $\opAt$ generates a $C^0$-group ${(\opSt(t))}_{t \in \R}$ and $\opBt$ is admissible.
Thus, \eqref{syst init p} is well-posed, that is, for every $\uzt \in L^2(0,L)$ and every $\Ut \in L^2(0,T)$, there exists a unique solution $\ut \in C^0([0,T];L^2(0,L))$ to \eqref{syst init p} and, in addition, there exists $C>0$ (which does not depend on $\uzt$ nor $\Ut$) such that
\begin{equation}\label{continu}
\norm{\ut}_{C^0([0,T];L^2(0,L))} \leq C \left(\norm{\uzt}_{L^2(0,L)}+\norm{\Ut}_{L^2(0,T)}\right).
\end{equation}

The following proposition shows that it is indeed equivalent to consider \eqref{syst init p} or \eqref{syst init} from a controllability point of view.
\begin{proposition}\label{prop equiv}
\eqref{syst init} is exactly controllable at time $T$ if, and only if, \eqref{syst init p} is exactly controllable at time $T$.
\end{proposition}

Roughly speaking, to prove Proposition \ref{prop equiv}, it suffices to take $\uzt=u^0$ and $U(t)=\ut(t,0)+\Ut(t)$.
We postpone the rigorous proof to \ref{app 2}.

In addition, note that
$$\ker(\lambda-\opA^*) \cap \ker \opB^*=\ker(\lambda-\opAt^*) \cap \ker \opBt^*,$$
for every $\lambda \in \C$.
As a result, \eqref{ucp} is equivalent to
\begin{equation}\label{ucp bis}
\ker(\lambda-\opAt^*) \cap \ker \opBt^*=\ens{0}, \quad \forall \lambda \in \C.
\end{equation}


\subsection{Bases and problem of moments in Hilbert spaces}\label{sect bases}

Let us recall here some basic facts about bases and the problem of moments in Hilbert spaces.
We follow the excellent textbook \cite{You}.
Let $H$ be a complex Hilbert space.
We say that $\ens{f_k}_{k \in \Z}$ is a basis in $H$ if, for every $f \in H$ there exists a unique sequence of scalar $\ens{\alpha_k}_{k \in \Z}$ such that $f=\sum_{k \in \Z} \alpha_k f_k$.
We say that $\ens{f_k}_{k \in \Z}$ is a Riesz basis in $H$ if it is the image of an orthonormal basis of $H$ through an isomorphism.
We can prove that $\ens{f_k}_{k \in \Z}$ is a Riesz basis if, and only if, $\ens{f_k}_{k \in \Z}$ is complete in $H$ and there exist $m,M>0$ such that, for every $N \in \N$, for every scalars $\alpha_{-N},\ldots,\alpha_N$, we have
\begin{equation}\label{equ Riesz}
m \sum_{k=-N}^{N} \abs{\alpha_k}^2 \leq \norm{\sum_{k=-N}^{N} \alpha_k f_k}_{H}^2 \leq M \sum_{k=-N}^{N} \abs{\alpha_k}^2.
\end{equation}
See e.g. \cite[Chapter 1, Theorem 9]{You}.

A useful criterion to prove that a sequence is a Riesz basis is the theorem of Bari (see e.g. \cite[Chapter 1, Theorem 15]{You}).
It states that $\ens{f_k}_{k \in \Z}$ is a Riesz basis of $H$ if $\ens{f_k}_{k \in \Z}$ is $\omega$-independent, that is, for every sequence of scalars $\ens{c_k}_{k \in \Z}$,
\begin{equation}\label{omega indep}
\sum_{k \in \Z} c_k f_k=0 \Longrightarrow \left(c_k=0, \quad \forall k \in \Z\right),
\end{equation}
and if $\ens{f_k}_{k \in \Z}$ is quadratically close to some orthonormal basis $\ens{e_k}_{k \in \Z}$ of $H$, that is
$$\sum_{k \in \Z} \norm{f_k-e_k}_{H}^2<+\infty.$$

On the other hand, we say that $\ens{f_k}_{k \in \Z}$ is a Bessel sequence in $H$ if, for every $f \in H$, we have
$$\sum_{k \in \Z} \abs{\ps{f}{f_k}{H}}^2<+\infty.$$
We can prove that $\ens{f_k}_{k \in \Z}$ is a Bessel sequence in $H$ if, and only if, $\ens{f_k}_{k \in \Z}$ satisfies the second inequality in \eqref{equ Riesz}.
See e.g. \cite[Chapter 2, Theorem 3]{You}.

Finally, we say that $\ens{f_k}_{k \in \Z}$ is a Riesz-Fischer sequence in $H$ if, for every sequence of scalars $\ens{c_k}_{k \in \Z}$ that belongs to $\ell^2(\Z)$, there exists (at least) a solution $f \in H$ to the problem of moments
$$c_k=\ps{f}{f_k}{H}, \quad \forall k \in \Z.$$
We can prove that $\ens{f_k}_{k \in \Z}$ is a Riesz-Fischer sequence in $H$ if, and only if, $\ens{f_k}_{k \in \Z}$ satisfies the first inequality in \eqref{equ Riesz}.
See e.g. \cite[Chapter 2, Theorem 3]{You}.

Observe then that, a Riesz basis is nothing but a complete Bessel and Riesz-Fischer sequence.
We refer to \cite[Chapter 4]{You} for more details on the problem of moments.

To prove Theorem \ref{thm cont}, the idea is to write the controllability problem as a problem of moments.
To achieve this goal, and to prove that the resulting problem of moments indeed has a solution, we first need to establish some spectral properties of our operator $\opAt^*$.

\subsection{Spectral properties of $\opAt^*$}

From now on, we assume that $\g$ depends only on its first variable $x$:
\begin{equation}\label{hyp gdx bis}
\g(x,y)=\g(x), \quad g \in L^2(0,L).
\end{equation}

The first proposition gives the basic spectral properties of $\opAt^*$.

\begin{proposition}\label{prop spec gdx}
Assume that \eqref{hyp gdx bis} holds.
Then,
\begin{enumerate}
\item\label{point 1}
For every $\lambda \in \C$, we have
$$\ker(\lambda-\opAt^*)=\ens{a e^{-\lambda x}+bw_{\lambda}(x) \quad \middle| \quad (a,b) \in \C^2, \quad H(\lambda) \begin{pmatrix} a \\ b \end{pmatrix}=0},$$
where we have introduced the matrix
$$
H(\lambda) =
\begin{pmatrix}
1-e^{-\lambda L} & -w_{\lambda}(L) \\
\ds \int_0^L \ext{\g(x)} e^{-\lambda x} \, dx & \ds \int_0^L \ext{\g(x)} w_{\lambda}(x) \, dx-1
\end{pmatrix},
$$
and the function
$$
w_{\lambda}(x) = \int_0^x e^{-\lambda(x-\sigma)} \, d\sigma
=\left\{\begin{array}{cl}
\ds \frac{1-e ^{-\lambda x}}{\lambda} & \mbox{ if } \lambda \neq 0, \\
\ds x & \mbox{ if } \lambda=0.
\end{array}\right.
$$
\item\label{point 2}
We have
$$
\spec{\opAt^*}=\ens{\lambda_k = \frac{2ik\pi}{L} \, \middle| \, k \in \Z, k \neq 0} \cup \ens{\lambda_0 = \int_0^L \ext{\g(x)} \, dx}.
$$
\end{enumerate}
\end{proposition}

\begin{proof}
Let us prove \ref{point 1}.
Let $\lambda \in \C$.
Let $z \in \ker(\lambda-\opAt^*)$, that is,
\begin{equation}\label{ODE}
\left\{
\begin{array}{c}
z \in H^1(0,L), \quad z(L)=z(0),\\
\ds  \lambda z(x)+z'(x)-\int_0^L \ext{\g(\sigma)}z(\sigma) \, d\sigma=0, \quad x \in (0,L).
\end{array}
\right.
\end{equation}
Solving the ODE in \eqref{ODE} yields
\begin{equation}\label{sol ODE}
z(x)=e^{-\lambda x}z(0) +w_{\lambda}(x) I,
\end{equation}
with
$$
I = \int_0^L \ext{\g(\sigma)}z(\sigma) \, d\sigma.
$$
From the boundary condition $z(L)=z(0)$ we obtain the relation
$$
\left(1-e^{-\lambda L}\right)z(0)-w_{\lambda}(L) I=0.
$$
To obtain a second relation, we mutiply \eqref{sol ODE} by $\ext{\g}$ and integrate over $(0,L)$, so that
$$
\left(\int_0^L \ext{\g(x)} e^{-\lambda x} \, dx\right)z(0)+ \left(\int_0^L \ext{\g(x)} w_{\lambda}(x) \, dx-1\right)I=0.
$$

Conversely, let
$$z(x)=a e^{-\lambda x}+bw_{\lambda}(x),$$
where $(a,b) \in \C^2$ is such that
\begin{equation}\label{Hlambda}
H(\lambda) \begin{pmatrix} a \\ b \end{pmatrix}=0.
\end{equation}
Clearly, $z \in H^1(0,L)$.
From the first equation of \eqref{Hlambda} and $w_{\lambda}(0)=0$, we have $z(L)=z(0)$.
From the second equation of \eqref{Hlambda}, $z$ solves the ODE in \eqref{ODE}.

Let us now turn out to the proof of \ref{point 2}.
The map
$$
\begin{array}{rccc}
& \ker H(\lambda) &\longrightarrow &\ker(\lambda-\opAt^*) \\
&\begin{pmatrix} a \\ b \end{pmatrix} &\longmapsto & ae^{-\lambda x} +b w_{\lambda}(x),
\end{array}
$$
is an isomorphism (the injectivity can be seen using $w_{\lambda}(0)=0$).
As a result,
$$
\dim \ker(\lambda-\opAt^*)=\dim \ker H(\lambda), \quad \forall \lambda \in \C.
$$
In particular,
$$\lambda \in \spec{\opAt^*} \Longleftrightarrow \det H(\lambda)=0.$$
Let us now compute more precisely $\det H(\lambda)$.
Observe that
$$
1-e^{-\lambda x}-\lambda w_{\lambda}(x)=0, \quad \forall \lambda \in \C, \forall x \in [0,L].
$$
Thus, adding $\lambda$ times the second column of the matrix $H(\lambda)$ to its first column, we obtain
$$\det H(\lambda)=
\det
\begin{pmatrix}
0 & \ds -w_{\lambda}(L) \\
\ds \int_0^L \ext{\g(x)} \, dx -\lambda & \ds \int_0^L \ext{\g(x)} w_{\lambda}(x) \, dx-1
\end{pmatrix},
$$
so that
$$\det H(\lambda)=w_{\lambda}(L)\left(\int_0^L \ext{\g(x)} \, dx -\lambda\right).$$
Finally, from the very definition of $w_{\lambda}$, we can check that
$$
w_{\lambda}(L)=0 \Longleftrightarrow \lambda \in \ens{\frac{2ik\pi}{L} \, \middle| \, k \in \Z, k \neq 0}.
$$
\end{proof}

\begin{remark}\label{rem disj}
In view of the controllability, we shall always assume that
\begin{equation}\label{vp disj}
\lambda_0 \neq \lambda_k, \quad \forall k \neq 0.
\end{equation}
Indeed, if \eqref{vp disj} does not hold, then $\lambda_0$ is an eigenvalue of geometric multiplicity at least two and \eqref{syst init p} is then impossible to control since the control operator is one-dimensional.
This follows from the general inequality
$$\dim \ker(\lambda-\opAt^*) \leq \dim \range \opBt^*, \quad \forall \lambda \in \C,$$
which is a consequence of \eqref{ucp bis} (and we recall that \eqref{ucp bis} is a necessary condition to the controllability, see Remark \ref{rem ac}).
Note that \eqref{vp disj} holds in particular if $\g$ is a real-valued function.
\end{remark}

Under assumption \eqref{vp disj} it is not difficult to see that the eigenspaces of $\opAt^*$ can be rewritten as
$$
\ker(\lambda_k-\opAt^*)=\Span \ens{\phi_k},
$$
where
\begin{equation}\label{def phik}
\phi_0(x)=1, \quad
\phi_k(x)=e^{-\lambda_k x}+\frac{1}{\lambda_k-\lambda_0}\int_0^L \ext{\g(x)} e^{-\lambda_k x} \, dx.
\end{equation}

Let us now write the property \eqref{ucp bis} more explicitely for the case \eqref{hyp gdx bis} (the proof is straightforward thanks to \eqref{def phik}).
\begin{proposition}
Assume that \eqref{hyp gdx bis} and \eqref{vp disj} hold.
Then, \eqref{ucp bis} is equivalent to
\begin{equation}\label{ucp gdx}
1+\frac{1}{\lambda_k-\lambda_0}\int_0^L \ext{\g(x)} e^{-\lambda_k x} \, dx \neq 0, \quad \forall k \neq 0.
\end{equation}
\end{proposition}

\begin{remark}\label{rem exist G}
Actually, \eqref{ucp gdx} has to be checked only for a finite number of $k$.
Indeed, \eqref{ucp gdx} always holds for $k$ large enough since
\begin{equation}\label{rem ucp gdx}
\frac{1}{\lambda_k-\lambda_0}\int_0^L \ext{\g(x)} e^{-\lambda_k x} \, dx \xrightarrow[k \to \pm \infty]{} 0.
\end{equation}

On the other hand, there exist functions $\g$ such that \eqref{ucp gdx} fails for an arbitrary large number of $k$.
Indeed, observe that for real-valued function $\g$, the equality
$$1+\frac{1}{\lambda_k-\lambda_0}\int_0^L \ext{\g(x)} e^{-\lambda_k x} \, dx = 0,$$
is equivalent to (taking real and imaginary parts)
$$
\left\{\begin{array}{rl}
\ds \int_0^L \g(x)\cos\left(\frac{2k\pi}{L}x\right) \, dx & \ds =\int_0^L \g(x) \, dx, \\
\ds \int_0^L \g(x)\sin\left(\frac{2k\pi}{L}x\right) \, dx & \ds = \frac{2k\pi}{L}.
\end{array}\right.
$$
For instance, for any $a_0 \in \R$ and any $N \geq 1$, the function
$$\g(x)=a_0+\frac{2}{L}\sum_{k=1}^{N} a_0 \cos\left(\frac{2k\pi}{L}x\right)+\frac{2}{L}\sum_{k=1}^{N} \frac{2k\pi}{L} \sin\left(\frac{2k\pi}{L}x\right),$$
satisfies these equalities for $k=1, \ldots, N$.
\end{remark}

The next and last proposition provides all the additional spectral properties required to apply the method of moments.

\begin{proposition}\label{prop moment}
Assume that \eqref{hyp gdx bis} and \eqref{vp disj} hold.
Then,
\begin{enumerate}
\item
The eigenfunctions $\ens{\phi_k}_{k \in \Z}$ of $\opAt^*$ form a Riesz basis in $L^2(0,L)$.

\item
If \eqref{ucp gdx} holds, then $\inf_{k \in \Z} \abs{\opBt^*\phi_k} >0$.

\item
The set of exponentials $\{e^{-\ext{\lambda_k} t}\}_{k \in \Z}$ is a Riesz basis in $L^2(0,L)$.
\end{enumerate}
\end{proposition}

\begin{proof}
\begin{enumerate}
\item
We will use the theorem of Bari previously mentioned.
Clearly, $\{\frac{1}{\sqrt{L}}\phi_k\}_{k \in \Z}$ is quadratically close to the orthonormal basis $\{\frac{1}{\sqrt{L}} e^{\frac{-2ik\pi}{L}x}\}_{k \in \Z}$.
To prove that $\{\frac{1}{\sqrt{L}} \phi_k\}_{k \in \Z}$ is $\omega$-independent, it suffices to take the inner product of the series in \eqref{omega indep} with each $e^{\frac{-2ik\pi}{L}x}$.

\item
From \eqref{rem ucp gdx} we have $\opBt^* \phi_k \xrightarrow[k \to \pm \infty]{} 1$ and by assumption $\opBt^* \phi_k \neq 0$ for every $k \in \Z$.

\item
Again, it suffices to notice that $\{\frac{1}{\sqrt{L}} e^{-\ext{\lambda_k} t}\}_{k \in \Z}$ is $\omega$-independent and quadratically close to $\{\frac{1}{\sqrt{L}} e^{\frac{2ik\pi}{L}t}\}_{k \in \Z}$.
\end{enumerate}
\end{proof}


\subsection{Proof of Theorem \ref{thm cont}}

Let us first recall the following fondamental relation between the solution to \eqref{syst init p} and its adjoint state:
\begin{equation}\label{sol transpo}
\ps{\ut(T)}{z}{L^2}-\ps{\uzt}{\opSt(T)^*z}{L^2}=\int_0^{T} \Ut(t)\ext{\opBt^*\opSt(T-t)^*z} \, dt, \quad \forall z \in L^2(0,L).\end{equation}

We have now everything we need to apply the method of moments and prove Theorem \ref{thm cont}.

\begin{proof}
We are going to write the null-controllability problem as a problem of moments.
From \eqref{sol transpo} we see that $\ut(L)=0$ if, and only if,
$$-\ps{\uzt}{\opSt(L)^*z}{L^2}=\int_0^{L} \Ut(t)\ext{\opBt^*\opSt(L-t)^*z} \, dt, \quad \forall z \in L^2(0,L).$$
Since $\ens{\phi_k}_{k \in \Z}$ is a basis, it is equivalent to
$$-\ps{\uzt}{\opSt(L)^*\phi_k}{L^2}=\int_0^{L} \Ut(t)\ext{\opBt^*\opSt(L-t)^*\phi_k} \, dt, \quad \forall k \in \Z.$$
Since $\phi_k$ are the eigenfunctions of $\opAt^*$, we have $\opSt(\tau)^*\phi_k=e^{\lambda_k \tau} \phi_k$ and, as a result,
$$-\ps{\uzt}{\phi_k}{L^2}=\int_0^{L} e^{-\ext{\lambda_k} t} \Ut(t)\ext{\opBt^*\phi_k} \, dt, \quad \forall k \in \Z.$$
Since $\opBt^* \phi_k$ is a nonzero scalar, this is equivalent to
\begin{equation}\label{mom 1}
c_k=\int_0^{L} e^{-\ext{\lambda_k} t}\Ut(t) \, dt, \quad \forall k \in \Z,
\end{equation}
where
\begin{equation}\label{mom 2}
c_k=-\frac{1}{\ext{\opBt^*\phi_k}}\ps{\uzt}{\phi_k}{L^2}.
\end{equation}
Now, \eqref{mom 1}-\eqref{mom 2} is a standard problem of moments, if the sequence $\ens{c_k}_{k \in \Z}$ belongs to $\ell^2(\Z)$.
Since $\delta=\inf_{k \in \Z} \abs{\opBt^*\phi_k} >0$ and $\ens{\phi_k}_{k \in \Z}$ is a Riesz basis (in particular, a Bessel sequence), $\ens{c_k}_{k \in \Z}$ indeed belongs to $\ell^2(\Z)$:
$$\sum_{k \in \Z} \abs{c_k}^2
\leq \frac{1}{\delta^2} \sum_{k \in \Z} \abs{\ps{\uzt}{\phi_k}{L^2}}^2
<+\infty.$$
Finally, since $\{e^{-\ext{\lambda_k} t}\}_{k \in \Z}$ is a Riesz basis (in particular, a Riesz-Fischer sequence), the problem of moments \eqref{mom 1}-\eqref{mom 2} has a solution $\Ut \in L^2(0,L)$ (see Section \ref{sect bases}).
\end{proof}

\begin{remark}
Since $\{e^{-\ext{\lambda_k} t}\}_{k \in \Z}$ is a Riesz basis, the solution $\Ut \in L^2(0,L)$ to the problem of moments \eqref{mom 1}-\eqref{mom 2} is actually unique.
This shows that, at least in the case \eqref{hyp gdx}, the control $U \in L^2(0,L)$ given by Proposition \ref{prop exist k} is unique (note the complete analogy with the case $\g=0$ for which the only null-control possible in the square $(0,L)\times(0,L)$ is $U=0$).
As a result, there is also only one solution to the kernel equation \eqref{equ hstar} with boundary conditions \eqref{H regul traduit} and \eqref{hstar final}.
\end{remark}

\appendix
\def\appendixname{Appendix }

\renewcommand\appendix{\par
  \setcounter{section}{0}%
  \setcounter{subsection}{0}%
  \setcounter{equation}{0}
  \gdef\thefigure{\@Alph\c@section.\arabic{figure}}%
  \gdef\thetable{\@Alph\c@section.\arabic{table}}%
 \gdef\thesection{\appendixname\@Alph\c@section}%
  \@addtoreset{equation}{section}%
  \gdef\theequation{\@Alph\c@section.\arabic{equation}}%
}

\setcounter{proposition}{0}
\section{Functions of $H^1(\Tau_-) \cap H^1(\Tau_+)$}\label{app 1}
\renewcommand\thesection{\Alph{section}}
This appendix gathers some properties of the functions of $H^1(\Tau_-) \cap H^1(\Tau_+)$.
We start with a characterization of the space $H^1(\Tau_+)$ (with an obvious analogous statement for $H^1(\Tau_-)$).
We recall that, by definition, $f \in H^1(\Tau_+)$ if $f \in L^2(\Tau_+)$ and $f_x, f_y \in L^2(\Tau_+)$, where $f_y \in L^2(\Tau_+)$ means that there exists $F \in L^2(\Tau_+)$ such that
$$\iint_{\Tau_+} f(x,y) \phi_y(x,y) \, dxdy=-\iint_{\Tau_+} F(x,y) \phi(x,y) \, dxdy, \quad \forall \phi \in \cinfc(\Tau_+).$$
Such a $F$ is unique and it is also denoted by $f_y$.

\begin{proposition}\label{prop caract hut}
Let $f \in L^2(\Tau_+)$.
The two following properties are equivalent:
\begin{enumerate}
\item
$f_y \in L^2(\Tau_+)$.

\item
For a.e. $x \in (0,L)$, the map
$$f(x):y \longmapsto f(x,y),$$
belongs to $H^1(x,L)$ and
$$\iint_{\Tau_+} \abs{f(x)'(y)}^2 \, dydx <+\infty.$$
\end{enumerate}
Moreover, $f(x)'(y)=f_y(x,y)$.
\end{proposition}

With the help of Proposition \ref{prop caract hut} it is not difficult to establish the following.

\begin{proposition}\label{prop ipp}
Let $f \in H^1(\Tau_-) \cap H^1(\Tau_+)$ and let us denote by $f_- \in L^2(\partial \Tau_-)$ (\textit{resp.} $f_+ \in L^2(\partial \Tau_+)$) the trace on $\Tau_-$ (\textit{resp.} $\Tau_+$) of the restriction of $f$ to $\Tau_-$ (\textit{resp.} $\Tau_+$).
\begin{enumerate}

\item\label{hyp A2}
For every $\varphi \in H^1(0,L)$, for a.e. $x \in (0,L)$, we have
\begin{multline*}
\int_0^L f(x,y)\varphi'(y) \, dy=
-\int_0^L f_y(x,y)\varphi(y) \, dy 
+\Big(f_-(x,x)-f_+(x,x)\Big)\varphi(x)
-f_-(x,0)\varphi(0)
+f_+(x,L)\varphi(L).
\end{multline*}

\item\label{hyp A3}
For every $\varphi \in L^2(0,L)$, the map
$$\Phi: x \mapsto \int_0^L f(x,y) \varphi(y) \, dy,$$
is in $H^1(0,L)$ with derivative
$$\Phi'(x)=
\Big(f_-(x,x)-f_+(x,x)\Big)\varphi(x)+\int_0^L f_x(x,y) \varphi(y) \, dy,$$
and traces
$$\Phi(0)=\int_0^L f_+(0,y) \varphi(y) \, dy, \quad \Phi(L)=\int_0^L f_-(L,y) \varphi(y) \, dy.$$
\end{enumerate}
\end{proposition}

\renewcommand\thesection{Appendix \Alph{section}}
\section{Proof of proposition \ref{prop equiv sol}}\label{app proof prop}

This appendix is devoted to the proof of Proposition \ref{prop equiv sol}.

\begin{proof}
Let $u^0 \in L^2(0,L)$ be fixed.
Set $w^0=P^{-1}u^0 \in L^2(0,L)$ and let $w \in C^0([0,T];L^2(0,L))$ be the corresponding solution to \eqref{syst targ}.
Let us recall that this means that $w$ satisfies
\begin{multline}\label{weak sol w}
\int_0^{\tau} \int_0^L w(t,x)\ext{\Big(-\psi_t(t,x)+\psi_x(t,x)\Big)} \, dx dt 
+\int_0^L w(\tau,x) \ext{\psi(\tau,x)} \, dx
-\int_0^L w^0(x) \ext{\psi(0,x)} \, dx=0,
\end{multline}
for every $\psi \in C^1([0,\tau]\times[0,L])$ such that $\psi(\cdot,0)=0$, and every $\tau \in [0,T]$.
Note that, by density, it is equivalent to take test functions $\psi$ in $L^2(0,\tau;H^1(0,L)) \cap C^1([0,\tau];L^2(0,L))$.
Let $u$ be defined by
$$
u(t)=Pw(t).
$$
Since $w \in C^0([0,T];L^2(0,L))$, it is clear that
$$u \in C^0([0,T];L^2(0,L)).$$
Moreover, since $w(T)=0$, we also have
$$u(T)=0.$$
Let us now establish that $u$ is the solution to \eqref{syst fb}, that is it satisfies
\begin{multline}\label{equ u}
\int_0^{\tau} \int_0^L u(t,x)\ext{\left(-\phi_t(t,x)+\phi_x(t,x)-\int_0^L \ext{\g(y,x)}\phi(t,y) \, dy\right)} \, dx dt
\\
+\int_0^L u(\tau,x) \ext{\phi(\tau,x)} \, dx
-\int_0^L u^0(x) \ext{\phi(0,x)} \, dx
-\int_0^{\tau} \Gamma u(t)\ext{\phi(t,L)} \, dt=0,
\end{multline}
for every $\phi \in C^1([0,\tau]\times[0,L])$ such that $\phi(\cdot,0)=0$, and every $\tau \in [0,T]$.
Since $u=Pw$ and $u^0=Pw^0$ by definition, we have
\begin{multline*}
\int_0^{\tau} \ps{u(t)}{-\ddt \phi(t)-\opA^* \phi(t)}{L^2} \, dt
+\ps{u(\tau)}{\phi(\tau)}{L^2}
-\ps{u^0}{\phi(0)}{L^2}
\\
=
\int_0^{\tau} \ps{w(t)}{-\ddt P^*\phi(t)-P^*\opA^* \phi(t)}{L^2} \, dt
+\ps{w(\tau)}{P^*\phi(\tau)}{L^2}
-\ps{w^0}{P^*\phi(0)}{L^2}.
\end{multline*}
On the other hand, since $\phi \in L^2(0,\tau;\dom{\opA^*})$, we can use the hypothesis \eqref{rela fonda-explicit} so that
$$-P^*\opA^* \phi(t)=-A_0^*P^*\phi(t)+P^*\Gamma^*\opB^*\phi(t).$$
It follows that
\begin{multline*}
\int_0^{\tau} \ps{u(t)}{-\ddt \phi(t)-\opA^* \phi(t)}{L^2} \, dt
+\ps{u(\tau)}{\phi(\tau)}{L^2}
-\ps{u^0}{\phi(0)}{L^2}
\\
=
\int_0^{\tau} \ps{w(t)}{-\ddt P^*\phi(t)-\opA_0^*P^* \phi(t)}{L^2} \, dt
+\ps{w(\tau)}{P^*\phi(\tau)}{L^2}
-\ps{w^0}{P^*\phi(0)}{L^2}
+\int_0^{\tau} \ps{w(t)}{P^*\Gamma^*\opB^*\phi(t)}{L^2}\, dt.
\end{multline*}
Taking the test function $\psi=P^*\phi$ in \eqref{weak sol w} (note that $\psi \in L^2(0,\tau;H^1(0,L))$ and satisfies $\psi(\cdot,0)=0$ since $P^*\left(\dom{\opA^*}\right) \subset \dom{\opA^*}$ by assumption), we see that the second line in the above equality is in fact equal to zero.
Taking the adjoints in the remaining term, we obtain \eqref{equ u}.
\end{proof}

\section{Controllability of \eqref{syst init p} and controllability of \eqref{syst init}}\label{app 2}
\renewcommand\thesection{\Alph{section}}
This appendix is devoted to the proof of Proposition \ref{prop equiv}.
The proof will use the following two lemmas.

\begin{lemma}\label{lemma 1 app}
Assume that
\begin{equation}\label{regu}
\uzt \in \dom{\opAt}, \quad \Ut \in H^1(0,T), \quad \Ut(0)=0.
\end{equation}
Then, the solution $\ut$ to \eqref{syst init p} belongs to $H^1((0,T)\times(0,L))$ and satisfies \eqref{syst init p} almost everywhere.
\end{lemma}

\begin{proof}
It follows from \eqref{regu} and the abstract result \cite[Proposition 4.2.10]{TW} that
\begin{equation*}
\ut \in C^1([0,T];L^2(0,L)).
\end{equation*}
On the other hand, by definition, we have
\begin{multline}\label{weaksol ut}
\int_0^{T} \int_0^L \ut(t,x)\ext{\left(-\phi_t(t,x)+\phi_x(t,x)-\int_0^L \ext{\g(y,x)}\phi(t,y) \, dy\right)} \, dx dt
-\int_0^L \uzt(x) \ext{\phi(0,x)} \, dx
-\int_0^{T} \Ut(t)\ext{\phi(t,L)} \, dt=0,
\end{multline}
for every $\phi \in C^1([0,T]\times[0,L])$ such that $\phi(t,L)=\phi(t,0)$ and $\phi(T,x)=0$.
In particular, for every $\phi \in \cinfc((0,T)\times(0,L))$, this gives
\begin{multline}\label{weak sol ut}
-\int_0^{T} \int_0^L \ut(t,x)\ext{\phi_t(t,x)} \, dx dt
+\int_0^{T} \int_0^L \ut(t,x)\ext{\phi_x(t,x)} \, dx dt 
-\int_0^{T} \int_0^L \ut(t,x)\ext{\left(\int_0^L \ext{\g(y,x)}\phi(t,y) \, dy\right)} \, dx dt
=0.
\end{multline}
On the other hand, since $\ut \in C^1([0,T];L^2(0,L))$, we have
$$\int_0^{T} \int_0^L \ut(t,x)\ext{\phi_t(t,x)} \, dx dt
=
-\int_0^{T} \int_0^L \ut_t(t,x)\ext{\phi(t,x)} \, dx dt.
$$
Coming back to \eqref{weak sol ut} we then obtain
$$
\int_0^{T} \int_0^L \ut(t,x)\ext{\phi_x(t,x)} \, dx dt =
\int_0^{T} \int_0^L \left(-\ut_t(t,x)+\int_0^L \ut(t,y) g(x,y) \, dy\right)\ext{\phi(t,x)} \, dx dt.
$$
Since the map
$$(t,x) \longmapsto -\ut_t(t,x)+\int_0^L \ut(t,y) g(x,y) \, dy,$$
belongs to $L^2((0,T)\times(0,L))$, this shows that $\ut_x \in L^2((0,T)\times(0,L))$ with
\begin{equation}\label{equ pp}
-\ut_x(t,x)=-\ut_t(t,x)+\int_0^L \ut(t,y) g(x,y) \, dy, \, \mbox{ for a.e. } t \in (0,T), \, x \in (0,L).
\end{equation}
Now, multiplying \eqref{equ pp} by $\ext{\phi} \in C^1([0,T]\times[0,L])$ such that $\phi(t,L)=\phi(t,0)$ and $\phi(T,x)=0$, integrating by parts and comparing with \eqref{weaksol ut}, we obtain
$$
\int_0^L \ut(0,x) \ext{\phi(0,x)} \, dx
+\int_0^{T} \ut(t,L)\ext{\phi(t,L)} \, dt
=\int_0^L \uzt(x) \ext{\phi(0,x)} \, dx
+\int_0^{T} \Ut(t)\ext{\phi(t,L)} \, dt.
$$
Taking $\phi(t,x)=\phi_1(t)\phi_2(x)$ with $\phi_1 \in C^{\infty}([0,T])$ such that $\phi_1(0)=1$ and $\phi_1(T)=0$, and $\phi_2 \in \cinfc(0,L)$, we obtain
$$\int_0^L \ut(0,x) \ext{\phi_2(x)} \, dx=\int_0^L \uzt(x) \ext{\phi_2(x)} \, dx.$$
Since $\cinfc(0,L)$ is dense in $L^2(0,L)$, this gives
$$\ut(0,x)=\uzt(x), \, \mbox{ for a.e. } x \in (0,L).$$
Similarly, we can prove that
$$\ut(t,L)=\Ut(t), \, \mbox{ for a.e. } t \in (0,T).$$
\end{proof}

\begin{lemma}\label{lemma 2 app}
Let $V=\dom{\opAt} \times \ens{\Ut \in H^1(0,T) \, \middle| \, \Ut(0)=0}$.
The map
\begin{equation}\label{map ext}
\begin{array}{ccl}
V & \longrightarrow & L^2(0,T) \\
\left(\uzt, \Ut\right)&\longmapsto& \ut(\cdot,0),
\end{array}
\end{equation}
where $\ut$ is the solution to \eqref{syst init p}, has a unique continuous extension to $L^2(0,L) \times L^2(0,T)$.
We shall keep the notation $\ut(\cdot,0)$ to denote this extension.
\end{lemma}

\begin{proof}
In virtue of Lemma \ref{lemma 1 app}, for $\left(\uzt,\Ut\right) \in V$, the map \eqref{map ext} is well-defined and \eqref{syst init p} is satisfied almost everywhere.
Multiplying \eqref{syst init p} by $(L-x)\ext{\ut}$, we obtain
\begin{multline*}
\int_0^T \frac{1}{2} \ddt \left(\int_0^L (L-x) \abs{\ut(t,x)}^2 \, dx\right) \, dt 
-\int_0^T \int_0^L (L-x) \frac{1}{2} \partial_x \left(\abs{\ut(t,x)}^2\right) \, dxdt \\
=\int_0^T \int_0^L \left(\int_0^L g(x,y) \ut(t,y) \, dy\right) (L-x)\ext{\ut(t,x)} \, dtdx.
\end{multline*}
Integrating by parts, this gives
\begin{multline*}
\frac{1}{2} \int_0^L (L-x) \abs{\ut(T,x)}^2 \, dx
-\frac{1}{2} \int_0^L (L-x) \abs{\uzt(x)}^2 \, dx \\
+\frac{1}{2} L \int_0^T \abs{\ut(t,0)}^2 \, dt 
-\frac{1}{2}\int_0^T \int_0^L \abs{\ut(t,x)}^2 \, dxdt 
=\int_0^T \int_0^L \left(\int_0^L g(x,y) \ut(t,y) \, dy\right) (L-x)\ext{\ut(t,x)} \, dtdx.
\end{multline*}
Using the inequality $ab\leq \frac{1}{2}a^2+\frac{1}{2}b^2$ (for $a,b \geq 0$) and the Cauchy-Schwarz's inequality, we can estimate the term on the right-hand side by $\norm{\ut}_{C^0([0,T];L^2(0,L))}^2$.
Using then \eqref{continu}, we obtain
$$
\int_0^T \abs{\ut(t,0)}^2 \, dt 
\leq C \left(\norm{\uzt}_{L^2(0,L)}^2+\norm{\Ut}_{L^2(0,T)}^2\right),
$$
for some $C>0$ (which does not depend on $\uzt$ nor $\Ut$).
As a result, the linear map \eqref{map ext} is continous on $L^2(0,L) \times L^2(0,T)$.
Since $V$ is dense in $L^2(0,L) \times L^2(0,T)$, we can extend this map in a unique continuous way to this space.
\end{proof}

We can now give the proof of Proposition \ref{prop equiv}:

\begin{proof}
Let $\uzt \in L^2(0,L)$ and $\Ut \in L^2(0,T)$.
Let $\ut \in C^0([0,T];L^2(0,L))$ be the corresponding solution to \eqref{syst init p}.
By density of $\dom{\opAt}$ in $L^2(0,L)$ and of $\cinfc(0,T)$ in $L^2(0,T)$, there exist sequences
$$\uzt_n \in \dom{\opAt}, \quad \Ut_n \in \cinfc(0,T),$$
such that
\begin{equation}\label{conv}
\uzt_n \xrightarrow[n \to +\infty]{} \uzt \, \mbox{ in } L^2(0,L), \qquad 
\Ut_n \xrightarrow[n \to +\infty]{} \, \Ut \mbox{ in } L^2(0,T).
\end{equation}
Let $\ut_n \in C^0([0,T];L^2(0,L))$ be the solution to
\begin{equation}\label{equ un}
\left\{
\begin{array}{rll}
\left(\ut_n\right)_t(t,x) -\left(\ut_n\right)_x(t,x) =&\ds \int_0^L \g(x,y)\ut_n(t,y) \, dy, & t \in (0,T) \, x \in (0,L), \\
\ut_n(t,L)-\ut_n(t,0)=& \Ut_n(t), & t \in (0,T), \\
\ut_n(0,x) =& \uzt_n(x), & x \in (0,L).
\end{array}
\right.
\end{equation}
By \eqref{continu} and \eqref{conv}, we have
$$\ut_n \xrightarrow[n \to +\infty]{} \ut \, \mbox{ in } C^0([0,T];L^2(0,L)).$$
On the other hand, by Lemma \ref{lemma 1 app}, we know that
$$\ut_n \in H^1((0,T)\times(0,L)),$$
and that \eqref{equ un} is satisfied almost everywhere.
Let $\tau \in [0,T]$ and $\phi \in C^1([0,\tau]\times[0,L])$ be such that $\phi(\cdot,0)=0$.
Multiplying \eqref{equ un} by $\ext{\phi}$ and integrating by parts yields
\begin{multline}\label{ipp un}
\int_0^{\tau} \int_0^L \ut_n(t,x)\ext{\left(-\phi_t(t,x)+\phi_x(t,x)-\int_0^L \ext{\g(y,x)}\phi(t,y) \, dy\right)} \, dx dt
\\
+\int_0^L \ut_n(\tau,x) \ext{\phi(\tau,x)} \, dx
-\int_0^L \uzt_n(x) \ext{\phi(0,x)} \, dx
-\int_0^{\tau} \left(\ut_n(t,0)+\Ut_n(t)\right)\ext{\phi(t,L)} \, dt
=0.
\end{multline}
By Lemma \ref{lemma 2 app} and \eqref{conv}, we know that
$$\ut_n(\cdot,0) \xrightarrow[n \to +\infty]{} \ut(\cdot,0) \, \mbox{ in } L^2(0,\tau).$$
Thus, passing to the limit $n \to +\infty$ in \eqref{ipp un}, we obtain
\begin{multline*}
\int_0^{\tau} \int_0^L \ut(t,x)\ext{\left(-\phi_t(t,x)+\phi_x(t,x)-\int_0^L \ext{\g(y,x)}\phi(t,y) \, dy\right)} \, dx dt
\\
+\int_0^L \ut(\tau,x) \ext{\phi(\tau,x)} \, dx
-\int_0^L \uzt(x) \ext{\phi(0,x)} \, dx
-\int_0^{\tau} \left(\ut(t,0)+\Ut(t)\right)\ext{\phi(t,L)} \, dt
=0.
\end{multline*}
This shows that $\ut$ is the (unique) solution of \eqref{syst init} with $u^0=\uzt$ and $U(t)=\ut(t,0)+\Ut(t)$.
\end{proof}

%




  \bibliographystyle{elsarticle-num} 
  \bibliography{biblio}






\end{document}